\documentclass[12pt]{amsart}

\usepackage{amsmath,amssymb,amsthm,amsfonts,verbatim,comment}
\usepackage{xcolor}
\usepackage{amsmath,amssymb,amsthm,amsfonts}
\usepackage{graphicx} 

\usepackage[
pdfborderstyle={},
pdfborder={0 0 0},
pagebackref,
pdftex]{hyperref}

\theoremstyle{plain}
\newtheorem{theorem}{Theorem}[section]

\newtheorem{lemma}[theorem]{Lemma}

\newtheorem{corollary}[theorem]{Corollary}

\newcommand{\QI}{\operatorname{QI}}
\newcommand{\Map}{\operatorname{Map}}
\newcommand{\Sep}{\operatorname{Sep}}

\newcommand{\calA}{\mathcal{A}}

\newcommand{\calC}{\mathcal{C}}

\newcommand{\calD}{\mathcal{D}}

\newcommand{\calFS}{\mathcal{FS}}

\newcommand{\ZZ}{\mathbb{Z}}

\DeclareMathOperator{\id}{id}

\DeclareMathOperator{\NonSep}{NonSep}

\title{Hyperbolic spaces not quasi-isometric to curve complexes}
\author{Javier Aramayona}
\thanks{JA was supported by grant PGC2018-101179-B-I00, and acknowledges financial support from the Spanish Ministry of Science and Innovation, through the “Severo Ochoa Programme for Centres of Excellence in R\&D (CEX2019-000904-S and CEX2023-001347-S)”}

\email{javier.aramayona@icmat.es}
\author{Hugo Parlier}
\email{hugo.parlier@unifr.ch}
\author{Richard Webb}
\email{richard.webb@manchester.ac.uk}

\date{\today}

\begin{document}

\begin{abstract}
We identify a condition that prevents a hyperbolic space from being quasi-isometric to the curve complex of any non-sporadic surface. Our result applies to several hyperbolic complexes, including arc complexes, disk complexes, non-separating curve complexes,  (hyperbolic) pants complexes, and to free splitting complexes of free groups.
\end{abstract}

\maketitle

\section{Introduction}
A highly successful approach to the study of mapping class groups has been through the use of the various simplicial complexes built from curves and arcs on surfaces, with the curve complex as a central example. A celebrated theorem of Masur~and~Minsky \cite{MasurMinsky} asserts that the curve complex is (Gromov) hyperbolic, a result that has deep implications for mapping class groups \cite{BKMM, BehrstockMinsky}, Teichm\"uller spaces \cite{Brock,EMRlarge,Rafi}, and hyperbolic 3-manifolds \cite{Brock, BCM, MinskyHyp}. Since then, a lot of effort has been put into analyzing geometric properties of these complexes, and in particular their (non-) hyperbolicity. 

 The quasi-isometric rigidity of mapping class groups and curve complexes was established respectively in \cite{BKMM} and \cite{RS}. However, while there are several families of complexes defined on surfaces which are hyperbolic, it is not immediately obvious whether or not they are quasi-isometric.

The purpose of this note is to prove a result that prevents a hyperbolic space from being quasi-isometric to a curve complex. Write $S_{g,n}$ for the connected orientable surface of genus $g$ and with $n$ punctures, and recall that its {\em complexity} is the number $\xi(S_{g,n})= 3g-3+n$. We will say that a surface $S$ is {\em non-sporadic} if $\xi(S) \ge 2$. We will show: 

\begin{theorem}
Suppose that a hyperbolic geodesic metric space $X$ admits two isometries $F$ and $G$ such that $F$ is loxodromic, $G$ is not, and $\ZZ^2 \cong \langle F,G \rangle \subseteq \QI(X)$. Then $X$ is not quasi-isometric to the curve complex of any non-sporadic surface.
\label{thm:main}
\end{theorem} 

As a consequence, we will deduce the following:

\begin{corollary}
Suppose that $g\ge 1$, $n \ge 1$, and $(g,n) \ne (1,1)$. Then, the arc complex $\mathcal A(S_{g,n})$ is not quasi-isometric to the curve complex of a non-sporadic surface.
\label{cor:arcs}
\end{corollary}

We remark that there is natural inclusion of $\mathcal A(S_{g,n})$ into the {\em arc-and-curve complex} $\mathcal{AC}(S_{g,n})$ --itself quasi-isometric to the curve complex $\mathcal C(S_{g,n})$-- which is not a quasi-isometry unless $g=0$ or $(g,n) = (1,1)$ \cite[Lemma 5.11]{MasurSchleimer}. The above corollary asserts that, in fact, {\em no quasi-isometry exists}; this answers negatively a question that we learned from Saul Schleimer. 
 
In addition, we will obtain a large list of hyperbolic complexes of curves and/or arcs that cannot be quasi-isometric to curve complexes. More concretely: 

\begin{corollary}
\label{cor:main}
The following hyperbolic complexes are not quasi-isometric to the curve complex of a non-sporadic surface:
\begin{enumerate}
    \item The pants complex $\mathcal P(S)$ and the cut-system complex $\mathcal H(S)$, when $\xi(S) = 2$; 
    \item The separating curve complex $\Sep(S_{g,n})$, where $g\ge 1$ and $n\ge 3$; 
    \item The non-separating curve complex $\NonSep(S_{g,n})$, for $n\ge 2$;
    \item The disk complex $\mathcal D(S_{g,n})$, where $n\ge 3$; 
    \item The free splitting complex $\mathcal{FS}_n$ of a free group of rank $n\ge 3$.
\end{enumerate}
\end{corollary}

\subsection*{Acknowledgement}

We are grateful for the workshop Surfaces, Manifolds and Related Topics (SMaRT), held in Luxembourg in 2024, where this work started.

\section{Proof of Theorem~\ref{thm:main}}

For the sake of brevity, we will assume the reader is familiar with geodesic metric spaces, quasi-isometries, Gromov hyperbolicity, and mapping class groups. We refer to the monographs \cite{BH99,DK18,FarbMargalit} for a thorough treatment. 

Let $X$ be a geodesic metric space. We will write $\QI(X)$ for the {\em quasi-isometry group} of $X$, namely the group of all quasi-isometries of $X$ modulo the equivalence relation where $f \sim g\in\QI(X)$ if
\[\sup_{x\in X} d_X(f(x),g(x)) <\infty.\]
The following lemma is well known, see for example \cite[Exercise~8.23]{DK18}.

\begin{lemma} \label{lem:qiisom}
    Let $X$ and $Y$ be quasi-isometric geodesic metric spaces, via quasi-isometries $h\colon X \to Y$ and $\bar h \colon Y \to X$ that are {\em coarse inverses}, namely $h\circ \bar h \sim \id_Y $ and $\bar h \circ h \sim \id_X$. Then the map $\QI(X) \to \QI(Y)$  given by  \[f \mapsto \bar h\circ f \circ h \] is an isomorphism. We call it the \emph{induced isomorphism} between $\QI(X)$ and $\QI(Y)$. 
\end{lemma}

We say that a geodesic metric space $X$ is \emph{quasi-isometrically rigid} if, for every $f \in \QI(X)$, there exists an isometry $g:X\to X$ such that $f \sim g$. Rafi~and~Schleimer proved the following (see Theorem~7.1 and Corollary~1.1 of \cite{RS}):

\begin{theorem} \label{thm:RS}
    Let $S$ be a surface with $\xi(S)\geq 2$. Then the curve complex $\calC(S)$ is quasi-isometrically rigid. Moreover, any quasi-isometry of $\calC(S)$ is within a bounded distance from a unique isometry. \hfill $\square$
\end{theorem}

Let $X$ be a geodesic metric space. An isometry $f\colon X\to X$ is said to be {\em elliptic} if it has an orbit of bounded diameter. It is {\em loxodromic} if its {\em asymptotic translation length (or stable length)} 
 \[|f| = \lim_{n\rightarrow \infty} \frac{1}{n}d_X(x, f^n(x))\] is positive. The limit above exists by a well-known application of Fekete's Lemma, and the limit is independent of the choice of $x\in X$. It is also well known that an isometry is loxodromic if and only if there is some $\alpha>0$ such that for some (or, equivalently, any) $x\in X$, we have $d_X(x,f^n(x))\geq \alpha |n|$ for all $n\in\mathbb{Z}$. We say that $f$ is \emph{parabolic} if $f$ is neither loxodromic nor elliptic.

Suppose $X$ is a geodesic metric space quasi-isometric to the curve complex $\calC(S)$ of some surface with $\xi(S)\geq 2$, and fix a quasi-isometry $h\colon X\to \calC(S)$ with coarse inverse $\bar h$. If $f\colon X \to X$ is an isometry then, by Theorem~\ref{thm:RS},  $h\circ f \circ \bar h$ is within a bounded distance from some unique isometry $g_f$ of $\calC(S)$. With this notation, we have: 

\begin{lemma} \label{lem:type}
 $f$ is loxodromic if and only if $g_f$ is loxodromic.
\end{lemma}

\begin{proof}
By Lemma~\ref{lem:qiisom}, for every integer $n\in \ZZ$, the map $h\circ f^n \circ 
\bar h$ is an element of $\QI(\calC(S))$. Moreover, by Rafi--Schleimer's Theorem~\ref{thm:RS} above, $h$ is within a bounded distance from a \emph{unique} isometry $g_n: \calC(S) \to \calC(S)$. The uniqueness of $g_n$ implies that $g_n=g_1^n=g_f^n$.

Now, we claim that $g_f$ is either loxodromic or elliptic. To see this, suppose first that $\xi(S) \ge 2$ and $S\ne S_{1,2}$. Since $g_f$ is an isometry, work of Ivanov \cite{Ivanov}, Korkmaz \cite{Korkmaz}, and Luo \cite{Luo} implies that $g_f$ is induced by a unique element of the mapping class group $\Map(S)$, which, by slightly abusing notation, is also denoted by $g_f$. Now, the Nielsen--Thurston classification (see, e.g., \cite{FarbMargalit}) implies that $g_f$ is either periodic, reducible, or pseudo-Anosov. If $g_f$ is pseudo-Anosov,  Masur--Minsky \cite{MasurMinsky} proved that $g_f$ acts loxodromically on $\calC(S)$. If, on the other hand, $g_f$ is reducible or periodic,  then it acts  elliptically on $\calC(S)$. Hence the claim follows. 

Now, $f^n$ is an isometry for all $n\in \mathbb{Z}$, and hence the quasi-isometry constants of $h\circ f^n \circ 
\bar h$ do not depend on $n$. It is implicit in Theorem~\ref{thm:RS} and the work in \cite{RS}, that every quasi-isometry is not only a bounded distance from an isometry, but that the distance bound only depends on the quasi-isometry constants (and the underlying surface $S$). We therefore have that $h\circ f^n \circ \bar h$ is at most a uniform distance (independent of $n$) away from $g_f^n$. This, together with the fact that $h$ and $\bar h$ are coarse inverses, yields: 
\[d_X(x,f^n(x)) \approx d_{X}(\bar h \circ h (x), \bar h \circ g_f^n \circ h(x)) \approx d_{\calC(S)}(h(x), g_f^n\circ h(x)), \]
where $\approx$ denotes equality up to uniform additive and multiplicative errors. In particular, $f$ and $g_f$ are either both loxodromic, or both elliptic, i.e. they have the same \emph{type}. This concludes the proof for $S\neq S_{1,2}$.

    When $S= S_{1,2}$, there are automorphisms of $\calC(S_{1,2}$) that are not induced by mapping classes. Instead, we make use of the fact that $\calC(S_{1,2})$ and $\calC(S_{0,5})$ are isomorphic (via the {\em hyperelliptic involution}, see \cite{Luo}), and run the same argument as above.\end{proof}

We are finally in a position to prove the main result of this note. 

\begin{proof}[Proof of Theorem~\ref{thm:main}]
We may suppose that $X$ has infinite diameter, otherwise the result is immediate. Suppose, by contradiction, that there is a quasi-isometry $h \colon X \to \calC(S)$, where $\xi(S)\geq 2$, and fix a coarse inverse $\bar h$ for $h$. Then $h\circ F \circ \bar h$ and $h \circ G \circ \bar h$ are both elements of $\QI(\calC(S))$. By \cite[Theorem~7.1]{RS}, they are within a bounded distance from unique isometries $F'$ and $G'$ of $\calC(S)$, respectively, which commute and have infinite order by Lemma~\ref{lem:qiisom}.

Now, Lemma~\ref{lem:type} implies that $F'$ is loxodromic but $G'$ is not. If $\xi(S) \ge 2$ and $S\ne S_{1,2}$, then  $F'$ must be induced by a pseudo-Anosov mapping class, and $G'$ must be induced by an infinite-order reducible mapping class. But then $F'$ and $G'$ cannot commute. To see this, since $G'$ is of infinite order and reducible, its {\em canonical reduction system} \cite{BLM}  is a non-empty multicurve on $S$, which $F'$ must preserve since it commutes with $G'$.  This shows that $F'$ is elliptic, which is a contradiction.

Finally, if  $S=S_{1,2}$, then we use again that $\calC(S_{1,2})$ and $\calC(S_{0,5})$ are isomorphic, and so we can run the same argument as above. 
\end{proof}

 \section{Applications}
 In this section we prove that the complexes that appear in Corollary~\ref{cor:main} are not quasi-isometric to the curve complex of any non-sporadic surface.
  There is a common strategy behind the proofs, which we will phrase in terms of {\em combinatorial complexes}; there are  simplicial complexes whose simplices are given by uniformly finite sets of arcs and/or curves with uniformly bounded intersection.
Following Masur--Schleimer \cite{MasurSchleimer}, a {\em witness} for a geometric complex $X(S)$ is a subsurface $Y \subset S$ such that every vertex of $X(S)$ intersects $Y$. If $Y\neq S$, we say that $Y$ is a \emph{proper} witness. The following is a direct corollary of the work of Masur--Minsky \cite{MasurMinsky}: 

\begin{lemma}
  Let $S$ be a finite-type surface and $X(S)$ a geometric complex invariant under a subgroup $H<\Map(S)$. Suppose that $W$ is a witness for $X(S)$, and that $F\in H$  preserves $W$ and  its restriction to $W$ is pseudo-Anosov. Then $F$ acts loxodromically on $X(S)$. 
\end{lemma}

As an immediate consequence, one obtains: 

\begin{corollary}
In the situation above, suppose in addition that the witness $W$ is proper, and that the Dehn twist $G$ about some boundary component of $W$ belongs to $H$. Then $X(S)$ is not quasi-isometric to the curve complex of any non-sporadic surface.
 \label{cor:witness}
\end{corollary}

\begin{proof}
We may assume that the boundary component of $W$ chosen above cannot be a witness of $X(S)$, for otherwise an orbit of $\langle F,G\rangle$ would provide a quasi-isometric embedding of $\mathbb{Z}^2$ into $X(S)$ (Lemma~5.2 of \cite{MasurSchleimer}), implying that $X(S)$ is not hyperbolic, and therefore not quasi-isometric to any curve complex. So we can assume instead that there is some vertex of $X(S)$ which is disjoint from this boundary component, and is therefore fixed by the Dehn twist, proving that $G$ is elliptic. Therefore, the maps $F$ and $G$ satisfy the conditions of Theorem~\ref{thm:main} and the result follows. \end{proof}

In light of Corollary~\ref{cor:witness}, we now proceed to discuss each of the complexes listed in Corollary~\ref{cor:main}. 

\subsection{Arc complexes} The arc complex $\mathcal A(S)$ is $\Map(S)$-invariant and hyperbolic \cite{MasurSchleimer}. There is a natural inclusion $\mathcal A(S)$ into the arc-and-curve complex $\mathcal {AC}(S)$, which is a quasi-isometry exactly when the genus of $S$ is zero \cite{MasurSchleimer}. We remark that $\mathcal{AC}(S)$ is easily seen to be quasi-isometric to the curve complex $\calC(S)$. 

Therefore, we focus on the case when $S$ has positive genus and $\xi(S)\ge2$. In this case, the witnesses are precisely the subsurfaces that contain every puncture of $S$, and we take $W$ to be any proper such subsurface, and apply Corollary~\ref{cor:witness}.

\subsection{Pants complexes and cut-system complexes} We deal with the case of pants complexes, as the case of the cut-system complex is totally analogous.

The pants complex $\mathcal P(S)$ is $\Map(S)$-invariant, and is known to be Gromov hyperbolic if and only if $\xi(S)=2$ \cite{MasurSchleimer}. We can thus focus our attention to these two surfaces, for otherwise $\mathcal P(S)$ cannot be quasi-isometric to any curve complex.

Now, any non-annular subsurface of $S$ is a witness for $\mathcal P(S)$, and so it suffices to take $W$ to be any such proper subsurface.

\subsection{Separating curve complexes} The separating curve complex $\Sep(S)$ is $\Map(S)$-invariant, and is known to be Gromov-hyperbolic exactly when $S$ has at least three boundary components \cite{Vokes}. For these surfaces, one may take $W$ to be the complement of a non-separating curve (see Example~2.4 in \cite{Vokes}) and apply Corollary~\ref{cor:witness} again.  

\subsection{Non-separating curve complexes} The non-separating curve complex of $S_{g,n}$ is $\Map(S_{g,n})$-invariant, and is Gromov-hyperbolic whenever $g\ge 2$. When $n\le 1$, the natural inclusion into $\calC(S_{g,n})$ is a quasi-isometry. However, if $n\ge 2$, the witnesses are the subsurfaces whose genus is $g$ and therefore, in these cases, we can apply Corollary~\ref{cor:witness}, and deduce that such a non-separating curve complex cannot be quasi-isometric to any curve complex of a non-sporadic surface.

\subsection{Disk complexes and free-splitting complexes.} Let $\calD_r$ be the disk complex of $\#_r \mathbb S^2\times \mathbb S^1$, and $\calFS _r$ the free-splitting complex of the free group on $n$ elements. Writing $S=S_{1,r-1}$, there are natural $1$-Lipschitz maps $\calA(S)\to \calD_r \to \calFS _r$. To see this, observe that an essential arc $a$ of $S_{1,r-1}$ determines the essential disk $a\times [0,1]$ in $S\times [0,1]$; by doubling this disk we obtain a sphere inside $\#_r S^2\times S^1$ --since this manifold is homeomorphic to the double of  $S_{1,r-1}\times [0,1]$)--, and hence a vertex of the free splitting complex $\calFS _r$. 

Now, let $Y$ be a proper witness for $\calA(S)$, choose a pseudo-Anosov $F$ on $Y$, and let $G$ be the Dehn twist about some essential boundary component of $Y$. Then $F$ acts loxodromically on $\calA(S)$, while $G$ fixes some arc, and so acts elliptically on $\calA(S)$. Then $F$ and $G$ also define elements of the handlebody group, and elements of the outer automorphism group of $F_r$, by a similar $\times [0,1]$ or doubling construction.
    
Hamenst\"adt~and~Hensel \cite{HH} provided $1$-coarsely Lipschitz left inverses for the natural maps $\calA(S)\to \calD_r \to \calFS _r$. Therefore, the $\times[0,1]$ and doubling constructions are in fact quasi-isometric embeddings; in particular, $F$ acts loxodromically on these complexes (whereas we know $G$ acts elliptically already). We conclude as before: by Theorem~\ref{thm:main} that none of these complexes can be quasi-isometric to some curve complex of a non-sporadic surface.

\bibliography{bib}
\bibliographystyle{plain}
\vspace{0.5cm}

\end{document}